\documentclass{amsart}
\usepackage{etex}
\usepackage[all, cmtip]{xy}
\usepackage{amssymb,amsmath,amscd,xy,graphicx,textcomp,mathtools}
\usepackage{epsf} 
\usepackage{hyperref}

\hypersetup{
  colorlinks   = true, 
  urlcolor     = blue, 
  linkcolor    = blue, 
  citecolor   = blue 
}

\newtheorem{theorem}{Theorem}[section]
\newtheorem{lemma}[theorem]{Lemma}

\newtheorem{definition}[theorem]{Definition}

\newtheorem{remark}[theorem]{\it Remark}

\newtheorem{proposition}[theorem]{Proposition}

\xyoption{arrow}

\xyoption{matrix}

\def\SL{\mathrm{SL}}

\def\C{\mathbb{C}}
\def\cC{\mathcal{C}}
\def\R{\mathbb{R}}
\def\Z{\mathbb{Z}}

\def\P{\mathcal{P}}
\def\Q{\mathbb{Q}}
\def\SL{\textup{SL}}

\def\tree{\mathcal{T}}

\begin{document}
\title[Gorenstein projective coordinate rings]{The Gorenstein property for projective coordinate rings of the moduli of parabolic $\SL_2$-principal bundles on a smooth curve}

\author{Theodore Faust, Christopher Manon}
\thanks{The 2nd author was supported by NSF grant DMS 1500966}
\date{\today}

\begin{abstract}
Using combinatorial methods, we determine that a projective coordinate ring of the moduli of parabolic principal $\SL_2$-bundles on a marked projective curve is not Gorenstein when the genus and number of marked points are greater than $1.$  
\end{abstract}

\maketitle

\tableofcontents


\section{Introduction}

Let $C$ be a smooth, projective, complex curve of genus $g$ and let $\bar{p} = \{p_1, \ldots, p_n\} \subset C$ be $n$ distinct points, called \emph{marked points}.  We select a non-negative integer $L \in \Z_{\geq 0}$ called the \emph{level}, and for each point $p_i$ we choose a non-negative integer $r_i \in \Z_{\geq 0}$ such that $r_i \leq L$.  This data determines a line bundle $\mathcal{L}(\vec{r}, L)$ on the moduli $\mathcal{M}_{C, \bar{p}}(\SL_2)$ of parabolic $\SL_2$-principal bundles on $(C, \bar{p})$ (\cite{P}, \cite{KNR}, \cite{BL}).  We let $R_{C, \bar{p}}(\vec{r}, L)$ be the section ring of $\mathcal{L}(\vec{r}, L)$.  Our main theorem concerns the Gorenstein property for the algebra $R_{C, \bar{p}}(\vec{r}, L)$. 

\begin{theorem}\label{main}
Let $g \geq 2$. If  $L = 1, 2, 4$, $\vec{r}$ has at most one non-zero entry, and this entry is equal to $L$, then $R_{C, \bar{p}}(\vec{r}, L)$ is Gorenstein for a generic curve $(C, \bar{p})$.   If this condition is not satisfied, then $R_{C, \bar{p}}(\vec{r}, L)$ is not Gorenstein for all curves $(C, \bar{p})$. 
\end{theorem}

\noindent
Let $\mathcal{M}_{g, n}$ be the moduli of curves $(C, \bar{p})$ of genus $g$ with $n$ points.  The term \emph{generic} in Theorem \ref{main} means the result holds for all curves $(C, \bar{p})$ in the complement of a countable union of closed substacks of $\mathcal{M}_{g, n}$.  

There has been steady progress in understanding the commutative algebra of the rings $R_{C, \bar{p}}(\vec{r}, L)$. The main result in Abe's paper \cite{A} gives a bound on the degree of generators needed to present this ring in the case $\vec{r} = {\bf 0}$ for generic curves $C$.  The second author extends this result to bounds on generators and relations in \cite[Theorem 1.5]{M4}, and introduces the sheaf of the algebra of conformal blocks (see below) and the method of toric degeneration as tools to study the projective coordinate rings of the moduli of principal $G$-bundles for a general simple, simply-connected group $G$.  Belkale and Gibney \cite{BelkaleGibney} show that these rings are always finitely generated when $G$ is type $A$.  Results of this paper also recover the $\SL_2$ case of a result of Kumar and Narasimhan \cite{KN} which shows that the coarse moduli of $G$-bundles (not parabolic) is a Gorenstein variety.  We also mention an important result of Mehta and Ramadas \cite{MehtaRamadas} which shows that the local rings of the coarse moduli of parabolic principal $\SL_2$-bundles are all Cohen-Macaulay in positive characteristic.  We focus on the $\SL_2$ case in this paper, but our techniques could be extended to the $\SL_3$ case using results of the second author in \cite{M10}.  

If $R$ is a graded $\C$-algebra of dimension $d$ which is the homomorphic image of the polynomial ring $S = \C[x_1, \ldots, x_n]$,  $R$ is said to be \emph{Gorenstein} if its canonical module $\Omega(R) = Ext_S^{n-d}(R, S)$ is isomorphic to a shifted copy $R(a)$ of $R$ as a graded module (see \cite[Theorem 3.3.7]{BH}).  If $X = Proj(R) \subset \mathbb{P}^n_\C$, and $R$ is Gorenstein, $X$ is said to be \emph{arithmetically Gorenstein} (AG).   Arithmetically Gorenstein projective schemes have a number of nice geometric and combinatorial properties, in particular all of their singularities are rational.  Moreover, the Betti numbers of  $R$ (as an $S$-module) satisfy a type of Poincar\'e duality: $\beta_i(R) = \beta_{n- d - i}(R)$. 

We prove Theorem \ref{main} by applying Stanley's work on Hilbert series of Cohen-Macaulay and Gorenstein domains \cite{stanley78}  to the Hilbert function $\phi^{g, n}_{\vec{r}, L}$ of $R_{C, \bar{p}}(\vec{r}, L)$.   The graded components of $R_{C, \bar{p}}(\vec{r}, L)$ can be viewed as spaces of $\mathfrak{sl}_2$-\emph{conformal blocks} (see \cite{P}, \cite{M4}), allowing us to use tools from representation theory to compute their dimensions. In particular, the algebra $R_{C, \bar{p}}(\vec{r}, L)$ is a sub-algebra of the $\Z^{n+1}$-graded algebra $\mathcal{V}^\dagger_{C,\bar{p}}$ of $\mathfrak{sl}_2$-\emph{conformal blocks} (see \cite{M4}).  The multigraded Hilbert function $\psi_{g, n}$ of $\mathcal{V}^\dagger_{C, \bar{p}}$ depends only on $g$ and $n$, and we have $\psi_{g, n}(N\vec{r}, NL) = \phi^{g, n}_{\vec{r}, L}(N)$.  Consequently, the Hilbert functions of the $R_{C, \bar{p}}(\vec{r}, L)$ are constant as $(C, \bar{p})$ varies in the moduli $\mathcal{M}_{g, n}$. 

As a step to proving Theorem \ref{main} we prove that $\mathcal{V}^\dagger_{C, \bar{p}}$ is Gorenstein for generic $(C, \bar{p})$ in Section \ref{poly}.  The $\mathcal{V}^\dagger_{C, \bar{p}}$ are the fibers of a $\Z^{n+1}$-graded sheaf of algebras $\mathcal{V}^\dagger$ which extends to a flat sheaf of algebras on the Deligne-Mumford compactification $\mathcal{M}_{g, n}\subset \overline{\mathcal{M}}_{g, n}$ (\cite[Proposition 1.6]{M4}).  Special points of $\overline{\mathcal{M}}_{g, n}$ correspond to graphs $\Gamma$, possibly with loops and multiedges, with first Betti number $g$ and $n$ leaves, whose vertices have degree $1$ or $3$. Such a graph $\Gamma$ is called \emph{trivalent}, the sets of edges, leaves, and vertices of $\Gamma$ are denoted $E(\Gamma)$, $L(\Gamma)$, and $V(\Gamma)$, respectively, and $(C_\Gamma, \bar{p}_\Gamma)$ denotes the corresponding stable curve.   A toric degeneration of the fiber $\mathcal{V}^\dagger_{C_\Gamma, \bar{p}_\Gamma}$ to an affine semigroup algebra $\C[S_\Gamma]$ is also constructed in \cite[Theorem 1.3]{M4}.  Following this construction, $R_{C_\Gamma, \bar{p}_\Gamma}(\vec{r}, L)$ degenerates to an affine semigroup algebra $\C[S_{\Gamma, \vec{r}, L}]$ for a certain polytope $\mathcal{P}_\Gamma(\vec{r}, L)$.  Using these constructions, we are effectively able to study $\phi^{g, n}_{\vec{r}, L}$ and $R_{C, \bar{p}}(\vec{r}, L)$ using combinatorial methods. 

These techniques are similar to those used by the $2$nd author in \cite{M3}, where a classification of arithmetically Gorenstein moduli of weighted points on the projective line is given.  The projective coordinate rings studied in \cite{M3} are the subcase of the $g = 0$ case when $L$ is very large compared to the $r_i$ of the problem we study here (see \cite[Section 4]{M4}).  The main theorem of \cite{M3} is an indication that establishing a classification of Gorenstein $R_{\mathbb{P}^1, \bar{p}}(\vec{r}, L)$ is more subtle than the higher genus cases.   For $g = 0$ the graphs we use above are trivalent trees $\tree$. We let $n = |L(\tree)|$ be the number of leaves of such a tree. In \cite[Section 2]{M3} it is shown that any lattice point $\omega \in P_{\tree}(\vec{r}, L)$ corresponds to a unique planar multigraph $T_\tree(\omega)$ on $n$ vertices.  We let $N^\tree_{ij}(\omega)$ be the number of edges between the $i$-th and $j$-th vertices of $T_\tree(\omega)$.  We let $\vec{2}$ and $2_\tree$ denote the vector and weighting of the edges of $\tree$ where each entry is $2$, respectively.  

\begin{theorem}[\cite{M3}, Theorem 1.7]
Let $L$ be sufficiently large, then $R_{\mathbb{P}^1, \bar{p}}(\vec{r}, L)$ is Gorenstein if and only if the following hold:\\
\begin{enumerate}
\item for some $a > 0$ we can write $a\vec{r} = \vec{2} + \vec{R}$, where either $R_i = \sum_{j \neq i} R_j$ or the triangle inequality holds for some $R_i, R_j, R_k$ and $R_\ell = 0$ for $\ell \neq i, j, k$,\\
\item there exists a compatible tree $\tree$ such that $N^\tree_{ij}(\omega - 2_\tree) \geq n-4$ when this quantity is non-zero, where $\omega$ is the unique interior lattice point in $\P_\tree(a\vec{r}, aL)$.\\
\end{enumerate}
\end{theorem}

\section{The Gorenstein property}\label{background}

In this section we prove a technical result which allows us to obtain information about a general fiber $R_{C, \bar{p}}(\vec{r}, L)$ of a sheaf of algebras on the moduli $\overline{\mathcal{M}}_{g, n}$ from a degeneration $\C[P_{\Gamma}(\vec{r}, L)]$ of a special fiber $R_{C_\Gamma, \bar{p}_\Gamma}(\vec{r}, L)$ (see \cite[Corollary 1.10]{M4}). We also review the combinatorial characterization of the Gorenstein property for affine semigroup algebras.

\subsection{The Gorenstein property in flat families}

 Let $R = \bigoplus_{N \geq 0} R_N$ be a positively graded Noetherian algebra with $R_0 = \C$ and Krull dimension $\dim(R) = d$.  Recall that a set of homogeneous elements $\{a_1, \ldots, a_d\} \subset R$ is called a \emph{homogeneous system of parameters} (\emph{hsop}) if the quotient vector spaces $R/\C[a_1, \ldots, a_d]$ is finite-dimensional, where $\C[a_1, \ldots, a_d] \subset R$ is the subalgebra generated by $\{a_1, \ldots, a_d\}$.  We recall a characterization of the Cohen-Macaulay (CM) property for algebras $R$ using the Hilbert series $\Phi_R(t)$ of $R$, see \cite[Corollary 3.2, Theorem 4.4]{stanley78}.  

\begin{theorem}\label{stanleyCM}
For any algebra $R$ as above and an \emph{hsop} $\{a_1, \ldots, a_d\} \subset R$ with $deg(a_i) = f_i$, there is a coefficient-wise inequality:
\begin{equation}
\Phi_R(t) \leq \frac{\Phi_{R/\langle a_1, \ldots, a_d\rangle}(t)}{\prod_{i =1}^d (1-t^{f_i})},\\
\end{equation}

\noindent
where $\langle a_1, \ldots, a_d \rangle \subset R$ is the ideal generated by the \emph{hsop}.  This is a coefficient-wise equality if and only if $R$ is Cohen-Macaulay.   Moreover, if $R$ is Cohen-Macaulay, then $R$ is Gorenstein if and only if $\Phi_R(t^{-1}) = (-1)^dt^a\Phi_R(t)$ for some $a$. 
\end{theorem}

We let $\mathcal{U}$ be a reduced, irreducible, affine scheme over $\C$, and $\mathcal{R} = \bigoplus_{N \geq 0} \mathcal{R}_N$ be a sheaf of graded $\mathcal{O}_\mathcal{U}$-algebras.  We further assume that each graded component $\mathcal{R}_N$ is a coherent locally free sheaf, and we suppose that the fiber $\mathcal{R}_p = R$ at a closed point $p \in \mathcal{U}$ is a finite-type Cohen-Macaulay $\C$-algebra.  These assumptions immediately imply that the Hilbert series $\Phi_{\mathcal{R}_q}(t)$ and Hilbert functions $\phi_{\mathcal{R}_q}(N)$ are independent of the choice of point $q \in \mathcal{U}$. As an example, following \cite[Theorem 1.3, Corollary 1.10]{M4}, we can take $\mathcal{R}_p = R_{C_\Gamma, \bar{p}_\Gamma}(\vec{r}, L)$, and $\mathcal{R}_q = R_{C, \bar{p}}(\vec{r}, L)$ for curves $(C, \bar{p})$ varying in an affine patch $\mathcal{U} \to \overline{\mathcal{M}}_{g, n}$. 

\begin{proposition}\label{familyproperties}
Suppose that $\mathcal{R}$ is a graded sheaf of algebras with $\mathcal{R}_p = R$ a CM algebra generated in degree $\leq L$. Then there is a complement of a countable number of closed subsets $\mathcal{U}^\circ \subset \mathcal{U}$ such that the fiber $\mathcal{R}_q$ is generated in degree $\leq L$ and CM for any $q \in \mathcal{U}^\circ$.  Moreover, if $R$ is Gorenstein, this also holds for all $q \in \mathcal{U}^\circ$. 
\end{proposition}

\begin{proof}
Suppose that $R$ is generated in degrees $\leq L$, and let $\mathcal{R}' \subset \mathcal{R}$ be the sub $\mathcal{O}_\mathcal{U}$-algebra generated by these graded components. For any $N$, the quotient $\mathcal{R}_N/\mathcal{R}'_N$ is $0$ in an open neighborhood $p \in \mathcal{U}_N \subset \mathcal{U}$. As a consequence, $\mathcal{R}'_q = \mathcal{R}_q$ for $q \in \bigcap_{N \geq 0} \mathcal{U}_N = \mathcal{U}'$. Now we choose $s_i \in \Gamma(\mathcal{U}, \mathcal{R}_{f_i})$ with $s_i\!\!\mid_p = a_i \in R$.  Let $\mathcal{O}_\mathcal{U}[s_1, \ldots, s_d] \subset \mathcal{R}$ be the sub $\mathcal{O}_\mathcal{U}$-algebra generated by the $s_i$. The quotient sheaf $\mathcal{R}_N/\mathcal{O}_\mathcal{U}[s_1, \ldots, s_d]_N$ is $0$ at $p$ for $N >> 0$.  Using an argument similar to the above one, we can find $\mathcal{U}'' \subset \mathcal{U}'$ where the specializations $\{s_1\!\!\mid_q, \ldots, s_d\!\!\mid_q\} \subset \mathcal{R}_q$ form an \emph{hsop} for every $q \in \mathcal{U}''$.  For any $q \in \mathcal{U}''$ we have:
\begin{equation}\label{ineq1}
\Phi_R(t) = \Phi_{\mathcal{R}_q}(t) \leq \frac{\Phi_{\mathcal{R}_q/\langle s_1\!\mid_q, \ldots, s_d\!\mid_q\rangle}(t)}{\prod_{i =1}^d (1-t^{f_i})}.
\end{equation}

\noindent
It follows that $\phi_{\mathcal{R}_p/\langle s_1\!\mid_p, \ldots, s_d\!\mid_p\rangle}(N) = \phi_{R/\langle a_1, \ldots, a_d\rangle}(N) \leq \phi_{\mathcal{R}_q/\langle s_1\!\mid_q, \ldots, s_d\!\mid_q\rangle}(N)$ for all $q \in \mathcal{U}''$.  Recall that $\mathcal{U}''$ is dense, so by semicontinuity it follows that this inequality holds for all $q \in \mathcal{U}$.  Moreover, the set of points where $\phi_{R/\langle a_1, \ldots, a_d\rangle}(N) < \phi_{\mathcal{R}_q/\langle s_1\!\mid_q, \ldots, s_d\!\mid_q\rangle}(N)$ is closed. As a consequence there is a $\mathcal{U}^\circ \subset \mathcal{U}''$, also the complement of a countable number of closed subsets, where (\ref{ineq1}) is an equality.  Theorem \ref{stanleyCM} then implies that $\mathcal{R}_q$ is CM for all $q \in \mathcal{U}^\circ$.  Now Theorem \ref{stanleyCM} (namely \cite[Theorem 4.4]{stanley78}) implies that $\mathcal{R}_q$ is Gorenstein for any such $q$ if and only if $\mathcal{R}_p$ is Gorenstein. 
\end{proof}

\subsection{Affine semigroups and the Gorenstein property}

Stanley's result in \cite{stanley78} illustrates the principle that the combinatorics of a grading on a Cohen-Macaulay algebra can be used to determine if that algebra has the Gorenstein property.  The characterization of the Gorenstein property for a normal affine semigroup algebra is an extremal instance of this.  Let $\mathcal{C} \subset \R^n$ be a \emph{positive polyhedral cone pointed at the origin}, which is \emph{integral with respect to a lattice} $\mathcal{M} \subset \R^n$, defined as:\\

\begin{enumerate}
\item $\mathcal{C}$ is the intersection of a finite number of half spaces $H_{F_i} = \{v \in \R^n | F_i(v) \geq 0\}$, $F_i: \R^n \to \R$ is linear, and the extremal rays of $\mathcal{C}$ all contain a point from $\mathcal{M}$,\\
\item There is a linear function $F: \R^n \to \R$ such that $F(w) \geq 0$ for all $w \in \mathcal{C}$. \\
\end{enumerate}

We let $S_\cC$ be the affine semigroup $\cC\cap \mathcal{M}$.  By formally adjoining scalars from $\C$, we obtain the affine semigroup algebra $\C[S_\cC]$. The algebra $\C[S_\cC]$ is finitely generated \cite[Proposition 6.1.2]{BH}, normal \cite[Theorem 6.1.4]{BH}, and Cohen-Macaulay, \cite[Theorem 6.3.5]{BH}.  We let $int(\mathcal{C})$ denote the set of relative interior points of the cone $\mathcal{C}$. For the following see \cite[Corollary 6.3.8]{BH}.

\begin{proposition}\label{affinesemigroupgorenstein}
With the assumptions above, the affine semigroup algebra $\C[S_\cC]$ is Gorenstein if and only if the set $int(\mathcal{C}) \cap \mathcal{M}$ is equal to the set $\omega + S_\cC$ for some $\omega \in S_\cC$. 
\end{proposition}

\noindent
Rephrased, Proposition \ref{affinesemigroupgorenstein} says that in order to show that $\C[S_\cC]$ is Gorenstein we must be able to find some interior point $\omega$ such that $v - \omega \in S_\cC$ for any $v \in int(\mathcal{C}) \cap \mathcal{M}$.

A convex polytope $\P \subset \R^n$ is said to be a \emph{lattice polytope} with respect to $\mathcal{M} \subset \R^n$ if all vertices of $\mathcal{P}$ are in $\mathcal{M}$.  There is a pointed integral polyhedral cone $\mathcal{C}_{\P}$ associated to $\P$ obtained by taking the $\R_{\geq 0}$ span of the vertices of $\P\times \{1\} \subset \R^{n+1}$.   Identifying $\R^n \times \{L\} \subset \R^{n+1}$ with $\R^n$, the level set $\mathcal{C}_{\P} \cap \R^n\times \{L\}$ of this cone is isomorphic to the Minkowski sum $L \circ \P = \P + \cdots + \P$. The intersection $S_\P = \mathcal{C}_{\P} \cap \mathcal{M}\times\Z \subset \R^{n+1}$ is then a graded normal affine semigroup.  We let $\C[S_\P]$ denote the corresponding affine semigroup algebra.  Observe that the $N$-th value of the Hilbert function of $\C[S_\P]$ is the number of lattice points in $L \circ \P$.  The following is a consequence of Proposition \ref{affinesemigroupgorenstein}.

\begin{proposition}\label{gradedaffinesemigroupgorenstein}
 The algebra $\C[S_\P]$ is Gorenstein if and only if some Minkowski sum $a \circ \P$ contains a unique interior lattice point $\omega \in int(a\P) \cap \mathcal{M}$ such that for any interior point $v \in int(L \circ \P) \cap \mathcal{M}$ there is a lattice point $u \in (a - L)\circ \P\cap \mathcal{M}$ such that $u + \omega = v$. 
\end{proposition}

\noindent
The negative $-a$ of the degree of the unique interior generator $\omega$ is called the \emph{$a$-invariant} of $\C[S_\P]$. 


\section{The polytope $\mathcal{P}_{\Gamma}(\vec{r}, L)$}\label{poly}

Recall that a graph $\Gamma$ is \emph{trivalent} if each vertex of $\Gamma$ has degree $1$ or $3$ (recall also that we do not require any other restrictions on $\Gamma$; multiple edges and loops are allowed). A \emph{leaf-edge} is an edge of $\Gamma$ for which (at least) one of its vertices has degree $1$. We denote by $V(\Gamma)$, $L(\Gamma)$, and $E(\Gamma)$ the set of vertices, set of leaves, and set of edges of $\Gamma$, respectively. In \cite[Definitions 1.2 and 1.4]{M4} the polytope $\mathcal{P}_{\Gamma}(\vec{r}, L)$ is defined as follows.

\begin{definition}\label{polydef}

Let $\Gamma$ be a a trivalent graph with $n$ leaves $\{\ell_1, \ldots, \ell_n\}$, and let $\vec{r}=(r_1,\cdots,r_n)$ be an $n$-tuple of non-negative integers. For a non-negative integer $L$, $\mathcal{P}_{\Gamma}(\vec{r}, L)$ is the set of functions $w: E(\Gamma) \to \R$ which satisfy four conditions:\\

\begin{enumerate}
\item $w(e) \geq 0$ for all $e \in E(\Gamma)$,\\
\item (Triangle inequalities) For any three edges $e, f, g$ meeting at a vertex $|w(e) - w(g)| \leq w(f) \leq w(e) + w(g)$,\\
\item For any three edges $e, f, g$ meeting at a vertex $w(e) + w(f) + w(g) \leq 2L$,\\
\item For $\ell_i$ the $i$-th leaf-edge of $\Gamma$, $w(\ell_i) = r_i$.\\
\end{enumerate}

\noindent
The polytope $\mathcal{P}_{\Gamma}(L)$ is defined by the first three conditions above. 

\end{definition}

We let $\P_\Gamma = \P_\Gamma(1)$.  It can be verified from Definition \ref{polydef} that $L \circ \mathcal{P}_{\Gamma} = \mathcal{P}_{\Gamma}(L)$.   In the special case that $\Gamma$ has exactly one internal vertex and three leaves we say that $\Gamma$ is a \emph{trinode}, and we let $\mathcal{P}_3(L)$ denote the corresponding polytope.  It is straightforward to check that $\mathcal{P}_3(L)$ is the convex hull of $(0, 0, 0), (L, L, 0), (L, 0, L)$, and $(0, L, L)$ in $\R^3$.  These convex bodies are considered in relation to the lattice $\mathcal{M}_{\Gamma} \subset \R^{E(\Gamma)}$ defined by the requirement that $w(e) \in \Z$ for all $e \in E(\Gamma)$ and $w(e) + w(f) + w(g) \in 2\Z$ for any three edges meeting at a vertex.  Let $S_{\Gamma, \vec{r}, L}$, $S_{\Gamma, L}$, and $S_{\Gamma}$ denote the graded affine semigroups associated to $\P_\Gamma(\vec{r}, L)$, $\P_{\Gamma}(L)$, and $\P_\Gamma$, respectively. 


For any curve $(C, \bar{p})$, and any graph $\Gamma$ with first Betti number equal to the genus of $C$ and $|L(\Gamma)| = |\bar{p}|$, the graded domain $R_{C, \bar{p}}(\vec{r}, L)$ can be flatly degenerated to the affine semigroup algebra $\C[S_{\Gamma, \vec{r}, L}]$, and the Hilbert functions of these algebras agree (\cite[Theorem 1.3, Corollary 1.10]{M4}).  The following proposition reduces the proof of Theorem \ref{main} to understanding the Gorenstein property for one of the algebras $\C[S_{\Gamma, \vec{r}, L}]$.  We let $\beta_1(\Gamma)$ denote the first Betti number of a graph $\Gamma$ (i.e. the number of edges in the complement of a spanning tree of $\Gamma$).

\begin{proposition}\label{indep}
The Hilbert function of $\C[S_{\Gamma, \vec{r}, L}]$ only depends on $g = \beta_1(\Gamma)$ and $n = |L(\Gamma)|$.  For any graphs $\Gamma, \Gamma'$ with $\beta_1(\Gamma) = \beta_1(\Gamma')$ and $|L(\Gamma)| = |L(\Gamma')|$, $\C[S_{\Gamma, \vec{r}, L}]$ is Gorenstein if and only of if $\C[S_{\Gamma', \vec{r}, L}]$ is Gorenstein.  If $\C[S_{\Gamma, \vec{r}, L}]$ is Gorenstein, then $R_{C, \bar{p}}(\vec{r}, L)$ is Gorenstein for generic $(C, \bar{p}) \in \mathcal{M}_{g, n}$.  If $\C[S_{\Gamma, \vec{r}, L}]$ is not Gorenstein, then $R_{C, \bar{p}}(\vec{r}, L)$ is not Gorenstein for all curves $(C, \bar{p}) \in \mathcal{M}_{g, n}$.
\end{proposition}

\begin{proof}
For any graph $\Gamma$ with $g = \beta_1(\Gamma)$ and $n = |L(\Gamma)|$,  $\phi^{g, n}_{\vec{r}, L}(N)$ is the number of lattice points in $N \circ \P_\Gamma(\vec{r}, L)$.  The algebras $\C[S_{\Gamma, \vec{r}, L}]$ are Cohen-Macaulay, so by Theorem \ref{stanleyCM}, $\C[S_{\Gamma, \vec{r}, L}]$ is Gorenstein if and only if $\C[S_{\Gamma', \vec{r}, L}]$ is Gorenstein.  Now, by \cite[Proposition 3.3]{M4}, there is a sheaf of algebras over $\mathbb{A}^1_\C$ with fiber $\C[S_{\Gamma, \vec{r}, L}]$ at $0$ and $R_{C_\Gamma, \bar{p}_\Gamma}(\vec{r}, L)$ at all $c \neq 0$.  By Proposition \ref{familyproperties} it follows that $\C[S_{\Gamma, \vec{r}, L}]$ is Gorenstein if and only if all fibers of this family are Gorenstein.  Moreover, if any fiber $R_{C, \bar{p}}(\vec{r}, L)$ is Gorenstein, then $\C[S_{\Gamma, \vec{r}, L}]$ is Gorenstein as well, as these two algebras share the same Hilbert function.   Finally, if $R_{C_\Gamma, \bar{p}_\Gamma}(\vec{r}, L)$ is Gorenstein, then Proposition \ref{affinesemigroupgorenstein} implies that $R_{C, \bar{p}}(\vec{r}, L)$ is Gorenstein for generic $(C, \bar{p})$.  
\end{proof}

\subsection{Interior points}

 In light of Proposition \ref{affinesemigroupgorenstein} we determine a necessary and sufficient condition that a point is an interior point in $\P_\Gamma(L)$.

\begin{proposition}\label{interior}
A point $w \in \mathcal{P}_{\Gamma}(L)$ is an interior point if and only if the inequalities $(1)-(3)$ in Definition \ref{polydef} are strict on $w$.
\end{proposition}

\begin{proof}
The polytope $\mathcal{P}_{\Gamma}(L)$ can be constructed as an intersection of the polytope $\mathcal{P}_3(L)^{|V(\Gamma)|} \subset \R^{3|V(\Gamma)|}$ with a linear space. We construct $\mathcal{P}_{\Gamma}(L)$ in this way by first splitting each edge $e \in E(\Gamma)$ to create a forest of $|V(\Gamma)|$ trinodes, and we let $e_1, e_2$ denote the pair of edges created by splitting $e \in E(\Gamma)$.  The polytope $\mathcal{P}_{\Gamma}(L)$ is then the subset of those $\bar{w} \in \mathcal{P}_3(L)^{|V(\Gamma)|} \subset \R^{3|V(\Gamma)|}$ satisfying $\bar{w}(e_1) = \bar{w}(e_2)$.  

The interior points of $\mathcal{P}_3(L)^{|V(\Gamma)|}$ are precisely those $\bar{w}$ which make the above inequalities $(1) -(3)$ strict at each trinode.  In order to establish that interior points of $\mathcal{P}_{\Gamma}(L)$ have the same description it suffices to show that $\mathcal{P}_{\Gamma}(L)$ contains one interior point of $\mathcal{P}_3(L)^{|V(\Gamma)|}$.  We construct such a point $\omega$ by choosing $\epsilon \in \Q$ such that $0 < 3\epsilon < 2L$, and defining $\omega(e) = \epsilon$ for all $e \in E(\Gamma)$.  
\end{proof}

A similar argument  can be used to classify interior points of $\mathcal{P}_{\Gamma_{g, n}}(\vec{r}, L)$ for a special trivalent graph $\Gamma_{g, n}$, provided the parameters $\vec{r}$ are chosen strictly between $0$ and $L$.  We let $\Gamma_{g, n}$ denote the graph depicted in Figure \ref{NonGorenstein-8} with $n$ leaves and first Betti number $g$.

\begin{figure}[htbp]
\centering
\includegraphics[scale = 0.35]{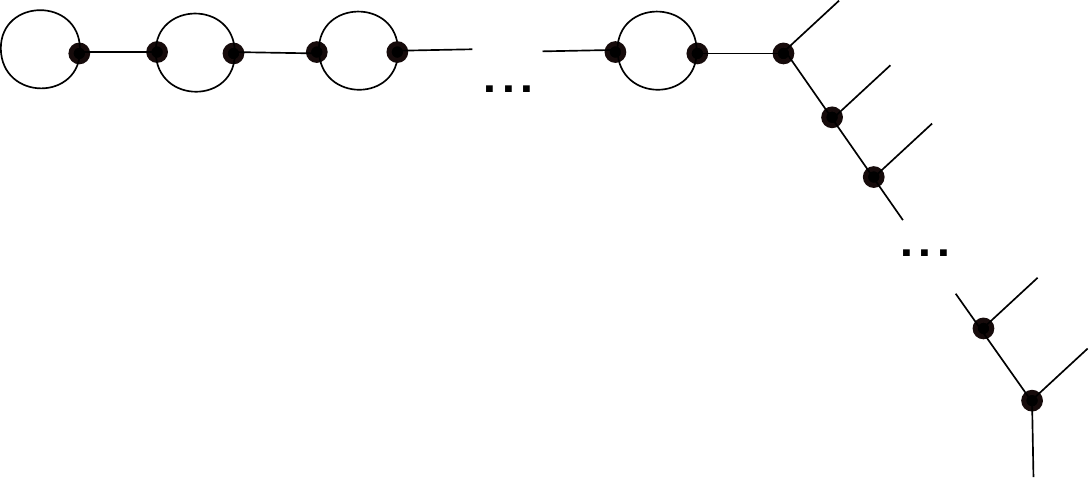}
\caption{The graph $\Gamma_{g, n}$}
\label{NonGorenstein-8}
\end{figure}

\begin{proposition}\label{fixedinterior}
If $\vec{r}=(r_1,\cdots,r_n)$ are any real numbers chosen such that $0 < r_i < L$, then a point $w \in \mathcal{P}_{\Gamma_{g, n}}(\vec{r}, L)$ is an interior point if and only if all inequalities in Definition \ref{polydef} are strict. 
\end{proposition}

\begin{proof}
It suffices to produce a point $w \in int(\mathcal{P}_{\Gamma_{g, n}}(L))$ which is also in $\mathcal{P}_{\Gamma_{g, n}}(\vec{r}, L)$.  Given two fixed lengths $0 < r \leq s < L$, we select $d$ such that $s-r < d < \textup{min}\{s,L-r\}$.  This choice ensures that all triangle inequalities on $d, r, s$ are strict and $d + r + s < 2L$.   By induction we can then find $d_1, \ldots, d_{n-2}$ so that $(r_1, r_2, d_1)$, $(d_1,r_3,d_2)$, $\ldots$, $(d_{n-3}, r_{n-1}, d_{n-2})$ are all in the interior of $\P_3(L)$.  Now we find any $0 < \epsilon <\textup{min}\{d, \frac{2}{3}L\}$, and assign $\epsilon$ to every other edge. 
\end{proof}

  Using Propositions \ref{interior} and \ref{fixedinterior} we give a classification of the interior lattice points of $\mathcal{P}_{\Gamma}(L)$ and $\mathcal{P}_{\Gamma_{g, n}}(\vec{r}, L)$.   Let $\omega_{\Gamma} \in \mathcal{P}_{\Gamma}(4)$ be the point which assigns $2$ to each edge of $\Gamma$. 

\begin{proposition}\label{interiorlattice}
Suppose that $0 < r_i < L$ for $1 \leq i \leq n$.  A lattice point $w \in \P_{\Gamma}(L)$ is in $int(\mathcal{P}_{\Gamma}(L))$ if and only if $w = \omega_{\Gamma} + u$ for some $u \in \P_{\Gamma}(L - 4)$.   A point $w \in \P_{\Gamma_{g, n}}(\vec{r}, L)$ is in $int(\mathcal{P}_{\Gamma_{g, n}}(\vec{r}, L))$ if and only if $w = \omega_{\Gamma_{g, n}} + u$ for some $u \in \P_{\Gamma_{g, n}}(\vec{r} - \vec{2}, L - 4)$, where $\vec{2}=(2,\cdots,2)$. 
\end{proposition}

\begin{proof}
By Proposition \ref{fixedinterior} the second claim follows from the first.  Suppose that $w \in \P_{\Gamma}(L)$ is an interior lattice point.  It follows that for any trinode $v \in V(\Gamma)$ with incident edges $e, f, g$, we must have $|w(f) - w(g)| < w(e) < w(f) + w(g)$ and $w(e) + w(f) + w(g) < 2L$.  The weak versions of these inequalities are still satisfied if $2$ is taken from each edge value and $4$ is taken from the level $L$.
\end{proof}

Proposition \ref{interiorlattice} is enough to establish the Gorenstein property for the affine semigroup algebra $\C[S_{\Gamma_{g, n}, L}]$ for certain values of $L$.



\begin{theorem}\label{acbGor}
The algebra $\C[S_{\Gamma, L}]$ is Gorenstein if and only if $L = 1, 2, 4$.  As a consequence, the algebra of conformal blocks $\mathcal{V}^{\dagger}_{C, \bar{p}}$ is Gorenstein for generic $(C, \bar{p})$. 
\end{theorem}

\begin{proof}
The set of lattice points in $\P_{\Gamma}(1)$ always has cardinality bigger than $1$, so it follows that $\mathcal{P}_{\Gamma}(L)$ contains interior points for any $L > 4$. Let $e, f, g \in E(\Gamma)$ meet at a vertex $v \in V(\Gamma)$, then if inequalities $(1)$ and $(2)$ in Definition \ref{polydef} are strict on a lattice point $w \in P_{\Gamma}(L)$, we must have $2 \leq w(e), w(f), w(g)$. It follows that $6 \leq w(e) + w(f) + w(g)$, so the polytope $\mathcal{P}_{\Gamma}(3)$ can not contain interior lattice points by Proposition \ref{interior}. However its $2$-nd Minkowski sum $\mathcal{P}_{\Gamma}(6)$ has $L > 4$, so it must contain more than $1$ interior point by Proposition \ref{interiorlattice}.  The cases $L =1, 2, 4$ all contain $\omega_{\Gamma}$ in their $4$-th, $2$-nd and $1$-st Minkowski sums respectively.  It follows that this point is a summand of all the interior points in the cones over these polytopes.   In particular, $\C[S_{\Gamma}]$ is Gorenstein by Propositions \ref{gradedaffinesemigroupgorenstein} and \ref{interiorlattice}, so by Proposition \ref{familyproperties}, $\mathcal{V}^\dagger_{C_\Gamma, \bar{p}_\Gamma}$ is Gorenstein, and $\mathcal{V}^{\dagger}_{C, \bar{p}}$ is Gorenstein for generic $(C, \bar{p}) \in \mathcal{M}_{g, n}$.
\end{proof}

The proof of Theorem \ref{acbGor} implies that the $a$-invariants of $\C[S_{\Gamma}]$, $\C[S_{\Gamma, 2}]$ and $\C[S_{\Gamma, 4}]$ are $-4, -2$, and $-1$ respectively.  In particular the algebra of conformal blocks $\mathcal{V}^{\dagger}_{C, \bar{p}}$ for generic $(C, \bar{p})$ always has $a$-invariant equal to $-4$, regardless of the genus of $C$ or number of marked points $n=|\bar{p}|$.



\subsection{Polyhedral symmetries}

Next we observe how a class of symmetries of the polytope $\mathcal{P}_{\Gamma}(L)$ affects the multigraded Hilbert function $\psi_{g, n}(\vec{r}, L)$.   Let $C_L: \R^3 \to \R^3$ be the affine map which sends $(x, y, z) \mapsto (L -x, L- y, z)$.

\begin{lemma}\label{sym}
The map $C_L$ permutes the lattice points of $\mathcal{P}_3(L)$ and defines a polyhedral symmetry.   
\end{lemma}

As a consequence of Lemma \ref{sym} a point $(x, y, z)$ is in $\P_3(L)$ if and only if $(L -x, L- y, z)$ is as well. The function $\psi_{0,3}(r_1, r_2, r_3, L)$ is either $1$ or $0$ depending on whether or not $(r_1, r_2, r_3) \in \P_3(L)$, so we conclude that $\psi_{0,3}(r_1, r_2, r_3, L) = \psi_{0,3}(L- r_1, L- r_2, r_3, L)$.   We recall three facts about $\psi_{g, n}(\vec{r}, L)$ which follow from properties of the algebra of conformal blocks $\mathcal{V}^{\dagger}_{C, \bar{p}}$ (see \cite[Proposition 2.3 and Proposition 3.1]{M4}).

\begin{proposition}\label{factorization}
The Hilbert functions $\psi_{g, n}(\vec{r}, L)$ satisfy the ``vacuum propagation'' identity:
\begin{equation}
\psi_{g, n+1}(\vec{r}, 0, L) = \psi_{g, n}(\vec{r}, L).\\
\end{equation}

\noindent
Any permutation $\sigma$ applied to $\vec{r}$ leaves the Hilbert function $\psi_{g, n}(\vec{r}, L)$ unchanged:
\begin{equation}
\psi_{g, n}(\vec{r}, L) = \psi_{g, n}(\sigma(\vec{r}), L).\\
\end{equation}

\noindent
The Hilbert functions $\psi_{g, n}(\vec{r}, L)$ have a so-called ``factorization'' property:
\begin{equation}
\psi_{g, n}(r_1, r_2, \ldots, r_n, L) = \sum_{0 \leq a \leq L} \psi_{g, n-1}(a, r_3, \ldots,r_n, L)\psi_{0, 3}(a, r_1, r_2, L).\\
\end{equation}

\end{proposition}

\begin{remark}
``Factorization'' as in Proposition \ref{factorization} above was used to define an interesting convolution operation on the Hilbert functions we consider here in \cite[Definition 3.30]{BW}. 
\end{remark}

Note that Proposition \ref{factorization} implies that the function $\phi^{g, n}_{\vec{r}, L}(N)$ does not depend on the order of $\vec{r}$, and $\phi^{g, n+1}_{\vec{r}, 0, L}(N) = \phi^{g, n}_{\vec{r}, L}(N)$.   We will need one further symmetry operation, which is captured in the following proposition.

\begin{proposition}\label{Sym}
For any $\vec{r} = (r_1, r_2, \ldots, r_n)$ and $L$, we have
\begin{equation}
\psi_{g, n}(r_1, r_2, \ldots, r_n, L) = \psi_{g, n}(L-r_1, L-r_2, r_3, \ldots, r_n, L),\\
\end{equation}
\begin{equation}
\phi^{g, n}_{r_1, r_2, \ldots, r_n, L}(N) = \phi^{g, n}_{L- r_1, L- r_2, r_3,\ldots, r_n, L}(N).\\
\end{equation}

\end{proposition}

\begin{proof}
The second equation follows from the first.  By Lemma \ref{sym} and Proposition \ref{factorization} we can compute:
\begin{equation}
\psi_{g, n}(r_1, r_2, \ldots, r_n, L) = \sum_{0 \leq a \leq L} \psi_{g, n-1}(a, \ldots, r_n, L)\psi_{0, 3}(a, r_1, r_2, L) =\\ 
\end{equation}
$$\sum_{0 \leq a \leq L} \psi_{g, n-1}(a, \ldots, r_n, L)\psi_{0, 3}(a, L- r_1, L- r_2, L) =  \psi_{g, n}(L-r_1, L-r_2, \ldots, r_n, L).$$

\end{proof}

\begin{remark}
Using  \cite[Equation 8]{StXu}, it is possible to deduce Proposition \ref{Sym} in the case $g = 0$ by applying permutations and the so-called Cremona transformations from the birational geometry of moduli spaces of rank $2$ parabolic vector bundles. 

\end{remark}

These symmetries, and the fact that the Gorenstein property depends only on the Hilbert function in this context, allows us to reduce the proof of Theorem \ref{main} to two cases. 

\begin{proposition}
Let $g \geq 2$ and $n \geq 1$, then the Hilbert function $\psi_{g, n}(\vec{r}, L)$ coincides
with $\psi_{g, n'}(\vec{r'}, L)$ where either all entries satisfy $0 < r_i' < L$ or $n' = 1$ and $r_1' = L$. 
\end{proposition}

The rest of the paper is dedicated to handling these two cases.

\section{The argument for generic $\vec{r}$}\label{gen}

Now we prove that the semigroup algebra $\C[S_{\Gamma_{g,n}, \vec{r}, L}]$ cannot be Gorenstein if $g \geq 2$, $n \geq 1$, and $0 < r_i < L$.  We argue as follows.  By Proposition \ref{gradedaffinesemigroupgorenstein}, if  $\C[S_{\Gamma_{g,n}, \vec{r}, L}]$ is Gorenstein, some Minkowski sum $M\circ \P_{\Gamma_{g, n}}(\vec{r}, L)$ contains a unique interior lattice point $v$.  By Proposition \ref{interiorlattice}, we must be able to write $v = u + \omega_{\Gamma_{g, n}}$, where $u \in \P_{\Gamma_{g, n}}(\vec{r} - \vec{2}, L -4)$.  Moreover, as $v$ is the unique interior lattice point of $\P_{\Gamma_{g, n}}(\vec{r}, L)$, $u$ must be the unique lattice point of  $ \P_{\Gamma_{g, n}}(\vec{r} - \vec{2}, L -4)$.  By Lemma \ref{morethanonelpoint} below this is not possible, in particular any $\P_{\Gamma_{g, n}}(\vec{r}, L)$ which contains one lattice point must contain two or more lattice points. Consequently, $\C[S_{\Gamma_{g, n}, \vec{r}, L}]$ is not Gorenstein by Proposition \ref{gradedaffinesemigroupgorenstein}.

\begin{lemma}\label{morethanonelpoint}
Let $n \ge 1$ and $g \ge 2$, then if $\P_{\Gamma_{g, n}}(\vec{r}, L)$ contains a lattice point, it contains at least $2$ lattice points. 
\end{lemma}

\begin{proof}
 
Let $e$ be the edge in $\Gamma_{g, n}$ which separates the tree portion from the portion with loops (see Figure \ref{NonGorenstein-8alt}).

\begin{figure}[htbp]
\centering
\includegraphics[scale = 0.35]{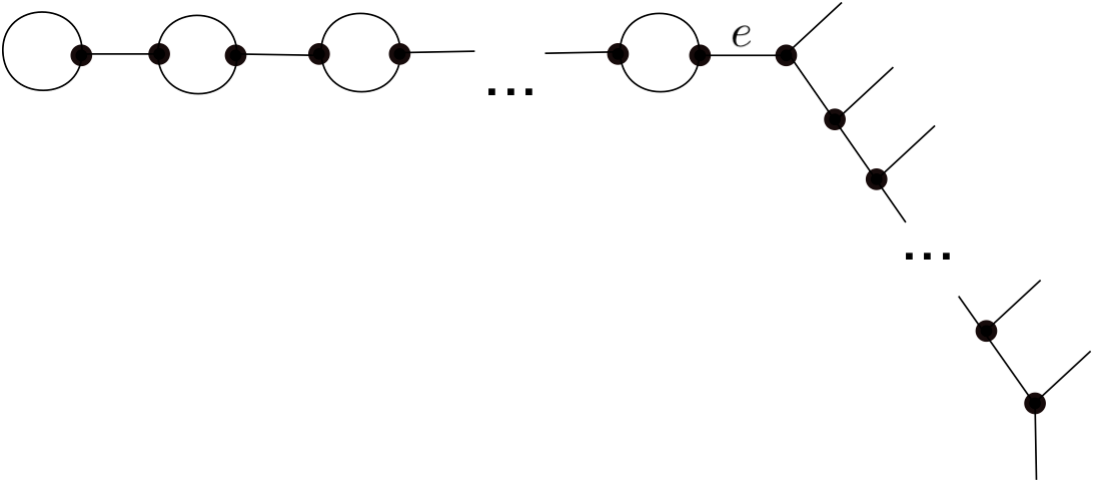}
\caption{}
\label{NonGorenstein-8alt}
\end{figure}

\noindent
Let $w$ be a lattice point in $\P_{\Gamma_{g, n}}(\vec{r}, L)$, and consider $w(e)$. If $w(e)$ is odd, then we are done, since the two edges adjacent to $e$ must have different parities (since $w(e)+x+y \in 2\Z$), hence we can swap them and create a distinct lattice point (as shown in Figure \ref{twolatticepts}).

\begin{figure}[htbp]
\centering
\includegraphics[scale=0.55]{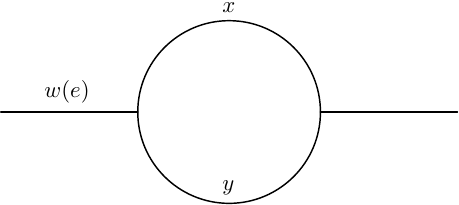}\\
\includegraphics[scale=0.55]{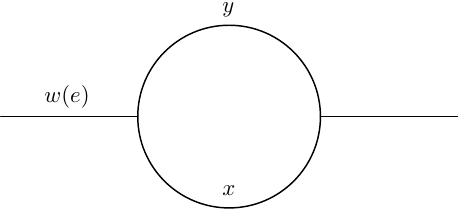}
\caption{}
\label{twolatticepts}
\end{figure}

\noindent
If $w(e)>0$ is even, then $w(e)\le L$ by the triangle inequalities (recall Definition \ref{polydef}). Hence, we can construct the  two graphs shown in Figure \ref{eventwolatticepts}, which clearly define distinct lattice points.

\begin{figure}[htbp]
\centering
\includegraphics{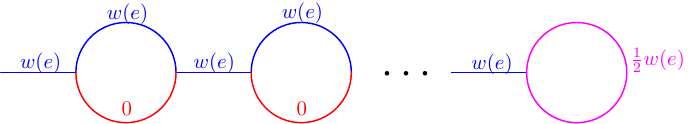}
\includegraphics{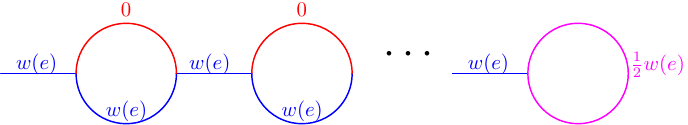}
\caption{}
\label{eventwolatticepts}
\end{figure}

\noindent

If $w(e)=0$, then we can create the graphs shown in Figure \ref{zerotwolatticepts}, which again describe distinct lattice points.

\begin{figure}[htbp]
\centering
\includegraphics[scale=0.5]{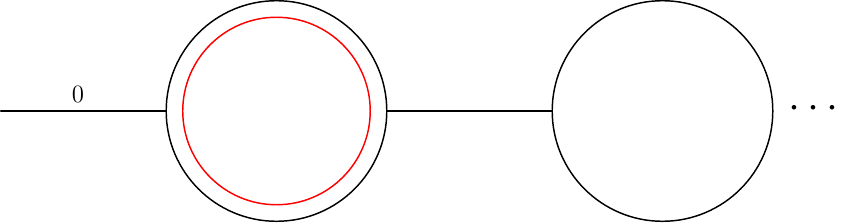}
\includegraphics[scale=0.5]{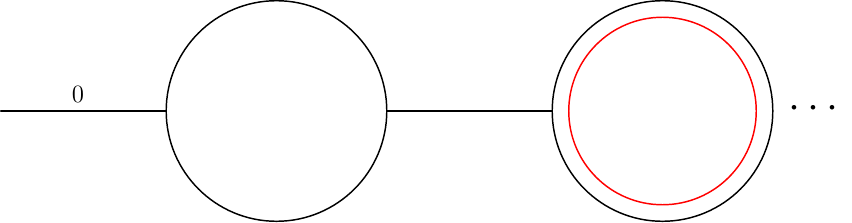}
\caption{}
\label{zerotwolatticepts}
\end{figure}

\noindent
Hence, given a lattice point $w \in P_{\Gamma_{g, n}}(\vec{r}, L)$, we can always find a distinct $w' \in P_{\Gamma_{g, n}}(\vec{r}, L)$. 
\end{proof}








 

\section{The case of one marked point}\label{1mark}

It remains only to check the part of Theorem \ref{main} captured in the following proposition.

\begin{proposition}
Let $C$ be a curve of genus $g \geq 2$ and $L =1, 2,$ or $4$. Then the algebra $R_{C, p}(L, L)$ is Gorenstein for generic curves $(C, p) \in \mathcal{M}_{g, 1}$.  If $L \neq 1, 2$, or $4$, then $R_{C, p}(L, L)$ is not Gorenstein for all curves $(C, p)$. 
\end{proposition}

\begin{proof}
We employ the graph $\Gamma_{g, 1}$ and we note that $\mathcal{P}_{\Gamma_{g,1}}(L, L)$ is isomorphic to the polytope $\mathcal{P}_{\Gamma_{g-1, 2}}(L)$ intersected with the hyperplane defined by requiring the sum of the leaf-weights to be $L$ (for example, see $\Gamma_{g,1}, \Gamma_{g-1,2}$ below in Figure \ref{3132} where $g=3$). Let $a$ and $b$ be the two leaf-edges of $\Gamma_{g-1, 2}$.  By taking any point $u$ with $u(a)$ and $u(b)$ nonzero, $u(a) + u(b) = L$, and all other edges assigned $\epsilon$ with $3\epsilon < 2L$, we construct an interior point of $\mathcal{P}_{\Gamma_{g-1, 2}}(L)$ which is also in $\mathcal{P}_{\Gamma_{g, 1}}(L, L)$.  It follows that interior points of the latter polytope are those which make all inequalities strict, and lattice interior points must have the point corresponding to $\omega_{\Gamma_{g-1, 2}}$ as a summand.  An argument similar to the proof of Theorem \ref{acbGor} then implies that this algebra is Gorenstein for $L = 1, 2$ and $4$. 

If $L\neq 1$ is odd, we are forced to assign to $a$ and $b$ different numbers $u(a) \neq u(b)$.  If $u$ were interior, then we could define another interior point $u'$ by setting $u'(e) = u(e)$ for $e \neq a, b$ and $u(a) = u'(b)$ and $u(b) = u'(a)$.  If $L$ is even and $> 4$, then we may change the label assigned to the end loop of $\Gamma_{g, 1}$ by $\omega_{\Gamma_{g-1, 2}}$ from $2$ to $3$ to obtain a second interior point.  
\end{proof}

\begin{figure}[htbp]
\centering
\includegraphics[scale=0.2]{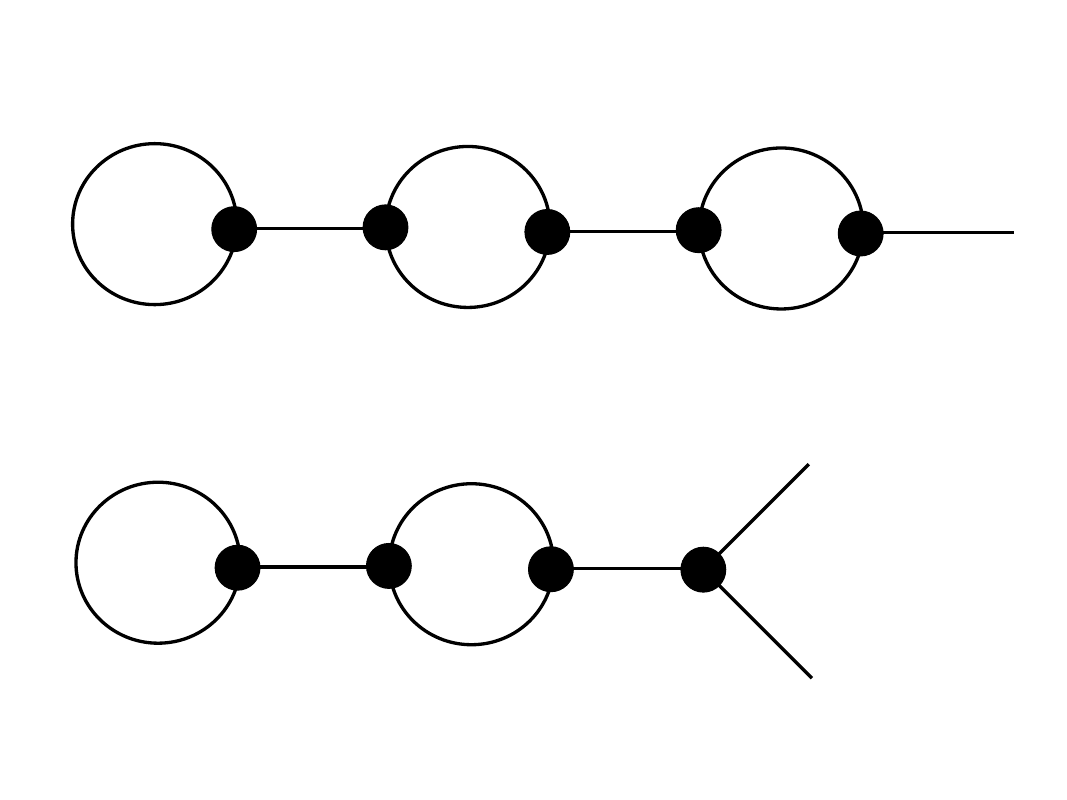}
\caption{The graphs $\Gamma_{3, 1}$ and $\Gamma_{2, 2}$.}
\label{3132}
\end{figure}

\bibliographystyle{alpha}
\bibliography{Biblio}

\bigskip
\noindent
Theodore Faust:\\
George Mason University\\ 
Fairfax, VA 22030 USA 

\bigskip
\noindent
Christopher Manon:\\
Department of Mathematics,\\ 
George Mason University,\\ 
Fairfax, VA 22030 USA

\end{document}